\documentclass[12pt]{article}
\usepackage{amsmath,amsfonts,amssymb,amsthm}
\usepackage[british]{babel}
\usepackage{color}
\usepackage{hyperref}
\usepackage{dsfont}

\newcommand{\F}{\mathcal{F}}
\newcommand{\calF}{\mathcal{F}}

\newcommand{\calS}{\mathcal{S}}
\newcommand{\calJ}{\mathcal{J}}
\newcommand{\pS}{\partial \mathcal{S}}

\newcommand{\DIV}{\textnormal{div}\,}
\newcommand{\eps}{\varepsilon}
\newcommand{\bbR}{\mathbb{R}}
\newcommand{\bbI}{\mathbb{I}}
\newcommand{\lra}{\longrightarrow}

\def\longrightharpoonup{\relbar\joinrel\rightharpoonup}
\def\cv{\stackrel{w}{\longrightharpoonup}}

\allowdisplaybreaks

\newtheorem{Theorem}{Theorem}
\newtheorem{Definition}{Definition}

\newtheorem{Lemma}{Lemma}
\newtheorem{Remark}{Remark}

\begin{document}
	
\date{\today}
\title{On the trajectory of a light small rigid body in an incompressible viscous fluid}
\author{Marco Bravin \footnote{Delft Institute of Applied Mathematics, Delft University of Technology, Mekelweg 4
2628 CD Delft, The Netherlands }, \v{S}\'arka Ne\v{c}asov\'a\footnote{Institute of Mathematics, Czech Academy of Sciences \v{Z}itn\'a 25, 115 67 Praha 1, Czech Republic.}}

\maketitle
	
\begin{abstract}
In this paper we study the dynamics of a small rigid body in a viscous incompressible fluid in dimension two and three. More precisely we investigate the trajectory of the rigid body in the limit when the its mass and its size tend to zero. We show that the velocity of the center of mass of the rigid body coincides with the background fluid velocity in the limit. We are able to consider the case where the density of the small rigid body is uniformly bounded respect to its size.
\end{abstract}

\section{Introduction}

In this paper we study the interaction of a small light rigid body with a incompressible viscous fluid in dimension two and three. The system fluid plus rigid body occupies the domain $ \bbR^d  $ for $ d = 2,3$.  The unknowns of the problem are the position of the rigid body $ \calS(t) \subset \Omega $ and the velocity of the fluid $ u_{\F} $ which is defined on the fluid domain $ \calF(t) = \bbR^d \setminus \calF(t) $ with values in $ \bbR^d $. Moreover the equations that $ u_{\F} $  satisfies,  are the incompressible Navier-Stokes equation
\begin{align}
\partial_t u_{\F} + \DIV( u_{\F} \otimes u_{\F}) - \nu \Delta u_{\calF}  -\nabla  p_{\F} =  & \, 0 \quad && \text{ for } x \in \calF(t),     \nonumber\\
\DIV( u_{\F}) = & \, 0 \quad && \text{ for } x \in \calF(t), \label{fluid:equ}\\
u_{\F} = \, & u_{\calS} \quad && \text{ for  } x \in \partial \calS(t),  \nonumber \\ 
|u_{\F}|  \lra \, & 0 \quad && \text{ for } |x| \in \lra +\infty, \nonumber
\end{align}        
where $ u_{\calS} $ is the velocity of the rigid body and $ \nu > 0 $ is the viscosity coefficient. Regarding the rigid body we assume that  it occupies the position $ \calS^{in} $ at initial time and it is a connected, simply-connected subset of $ \bbR^d $ with smooth boundary and  that it has constant density $ \rho_{\calS} \in \bbR $ with $ \rho_{\calS} \geq C > 0 $. 

The evolution of the rigid body is completely determined by the dynamic of the center of mass and of the angular rotation. Recall that   $ \calS^{in} $ has mass and center of mass defined respectively by 
\begin{equation*} 
m = \int_{\calS^{in}} \rho_{\calS} \, dx  \quad \text{ and } \quad h^{in} = \frac{1}{m}\int_{\calS^{in}} \rho_ {\calS} x\, dx.
\end{equation*}
If we denote by $ h(t) $ the position of the center of mass at time $ t $ and $ Q(t) $ the rotation around the center of mass respect to the initial configuration then
the position of the rigid body at time $ t $ is 
\begin{equation*}
\calS(t) = \left\{ y \text{ such that } Q^{T}(t)(y-h(t)) \in \calS^{in} \right\}
\end{equation*} 
and in dimension $ d = 3 $, its velocity is
\begin{equation}
\label{velocity}
u_{\calS} = \frac{d}{dt}(h(t) + Q(t)x) \Big|_{x = Q^{T}(t)(y-h(t))} = h'(t) + Q'(t)Q^T(t)(y-h(t))  = \ell(t) + \omega(t) \times  (x-h(t))
\end{equation}
where we denote $ \ell(t) = h'(t) $. Moreover from the fact that $ Q(t) $ is a rotation matrix, $ Q'(t)Q^T(t) $ is skew symmetric and it can be identify with a vector $ \omega(t) \in \bbR^3 $ through the relation
\begin{equation*} 
 Q'(t)Q^T(t) x = \omega(t) \times x
\end{equation*}
where $ x \in \bbR^3  $.  The evolution of $ \ell $ and $ \omega $ follows the Newton's laws that write
\begin{align}
m \ell'(t) = & \, - \oint_{\pS(t)} \Sigma(u_{\F}, p_{\F}) n \, ds \label{body:equ} \\
\calJ(t) \omega'(t) = & \, \calJ(t) \omega(t) \times \omega(t) - \oint_{\pS(t)} (x-h(t)) \times \Sigma(u_{\F}, \rho_{\F}) n \, ds \nonumber
\end{align}
where $ \Sigma $ is the stress tensor 
\begin{equation*}
\Sigma(u,p) = 2\nu D(u)- p \bbI \quad \text{ where } D(u)  = \frac{\nabla u + (\nabla u)^T}{2}
\end{equation*}
and $ \calJ $ is the inertia momentum and it is given through the formula
\begin{equation*}
\calJ(t) = \int_{\calS(t)} \rho_{\calS}\left [|x-h(t)|^2 \bbI - (x-h(t))\otimes (x-h(t))\right] \, dx = Q(t)\calJ(0)Q^{T}(t). 
\end{equation*}
Finally the initial conditions are 
\begin{equation}
 u_{\F}(0) = u_{\F}^{in}, \quad \ell(0) = \ell^{in} \quad \text{ and } \omega(0) = \omega^{in}, \label{in:cond}
 \end{equation}
such that they satisfy the compatibility conditions
\begin{equation}
\label{comp:cond}
\DIV(u_{\F}^{in})= 0 \text{ in } \F^{in} \quad \text{ and } \quad  u_{\F}^{in} = \ell^{in} + \omega^{in} \times (x-h^{in}) \quad \text{ in } \partial \calS^{in}.
\end{equation}
Moreover without loss of generality we have $ h(0) = 0 $ and $ Q(0) = 0 \bbI $.

In the case of dimension $ d = 2 $, equations \eqref{velocity} and  \eqref{body:equ} simplify, in fact $ Q'(t)Q^T(t) $ is skew symmetric and it can be identify with a scalar quantity $ \omega(t)  $ through the relation
\begin{equation*} 
 Q'(t)Q^T(t) x = \omega(t) \times x^{\perp}
\end{equation*}
where $ x = (x_1, x_2) \in \bbR^2  $ and $ x^{\perp} = (-x_2,x_1)^T $. In particular \eqref{velocity} becomes 
\begin{equation*}
u_{\calS} = \ell(t) + \omega(t)  (x-h(t))^{\perp}.
\end{equation*}
The Newton's laws write  
\begin{align*}
m \ell'(t) = & \, - \oint_{\pS(t)} \Sigma(u_{\F}, p_{\F}) n \, ds  \\
\calJ \omega'(t) = & \,- \oint_{\pS(t)} (x-h(t))^{\perp} \cdot \Sigma(u_{\F}, \rho_{\F}) n \, ds 
\end{align*}
and  $ \calJ $ is time independent.

System \eqref{fluid:equ}-\eqref{body:equ}-\eqref{in:cond} has been widely studied. The first works on the existence of Hopf-Leray type weak solutions  are \cite{Jud74} and  \cite{Serre} where they consider the case $ \Omega = \bbR^3 $. These results were then extended in \cite{GLS00}-\cite{CST}-\cite{DE}-\cite{Fei}-\cite{Fei1}.  Uniqueness was shown in \cite{GS15}  in dimension two and  in \cite{MNR} in dimension three under Prodi-Serrin conditions. Regularity was studied in dimension three under Prodi-Serrin conditions in \cite{MNR1}.  Well-posedness of strong solutions in Hilbert space setting was proved in \cite{GM}-\cite{Tak}-\cite{TT} and in the Banach space setting in \cite{GGH}-\cite{MT}. Notice that similar results holds in the case the Navier-Slip boundary conditions are prescribed on $\partial \calS $, see \cite{PS}-\cite{GH}-\cite{IO1}-\cite{BGMN}-\cite{CNM}.

Let now introduce  a small parameter $ \eps > 0 $ and for $ \bar{x} \in \bbR^d $, let $ \calS^{in}_{\eps} \subset  B_{\eps}(\bar{x}) $ a sequence of initial positions of the rigid body. In this paper we study the dynamics of the rigid body when $ \eps $ converges to zero. Let us recall that it has already  been shown that the fluid does not interact with a small rigid body see for instance \cite{LT}, \cite{He:2D}, \cite{He:3D} and \cite{FRZ2} under some mild assumptions on $ m_{\eps} $ and $ \calS_{\eps}^{in}$. Similar results are available also for compressible fluid see \cite{ION}, \cite{FRZ} and Section 6 of \cite{FRZ2}. Moreover we proved in \cite{IO8} that under some lower bounds on the masses 
$$ m_{\eps}/\eps^{1/2} \lra +\infty  \text{ for d = 3} \quad \text{ and } \quad m_{\eps} \geq C > 0 \text{ for d = 2} $$
the small rigid body will move with constant initial velocity. In this paper we show that under the assumptions
\begin{gather}
m_{\eps}/| \calS^{in}_ {\eps}|^{1/3}  \lra 0 \quad \text{ and }\quad |\calS_{\eps}^{in}| / \eps^{9/2} \lra +\infty  \text{ if d = 3}, \nonumber \\
m_{\eps}/| \calS^{in}_ {\eps}|^{\delta}  \lra 0 \quad \text{ and } \quad |\calS_{\eps}^{in}| / \eps^{\tilde{\delta}} \lra +\infty  \text{ for some } \delta > 0 \text{ and } 0 < \tilde{\delta} < 1 \text{ if d = 2} \label{ass}
\end{gather}
 and appropriate convergence of the initial data, the small rigid body will follow the fluid velocity.

This result show a different behaviour respect to the case of an inviscid incompressible fluid where the small rigid body is not accelerate by the fluid, see Section 1.4 of \cite{Munnier}. In the proof of our result we take advantage of the viscous term in an  essential way. The idea is to use a relative energy inequality where we estimate
\begin{align*}
\int_{\bbR^d} \rho_{\eps}|u_{\eps}(t,.) - R_{\eps}(u(t,.))|^2 + 4 \nu & \int_{0}^{T} \int_{\bbR^d} |D(u_{\eps}- R_{\eps}(u))|^2  \leq \int_{\bbR^d} \rho_{\eps}|u^{in}_{\eps} - R_{\eps}(u^{in})|^2 \\ & +  C \int_0^T \int_{\bbR^d} \rho_{\eps}|u_{\eps} - R_{\eps}(u)|^2+ Rest_{\eps},
\end{align*}
where $ R_{\eps} $ is a restriction operator that maps solenoidal vector fields to solenoidal vector fields  that are rigid in $ \calS_{\eps} $, i.e. with zero symmetric gradient is the part of the domain occupied by the rigid body, and $ Rest_{\eps} $ a small reminder. 

To estimate $ \ell_{\eps} - u_{\eps}(t, h_{\eps}(t)) $, it is possible to study $ \int_{\bbR^d} \rho_{\eps}|u_{\eps}(t,.) - R_{\eps}(u(t,.))|^2 \geq m_{\eps}| \ell_{\eps} - u_{\eps}(t, h_{\eps}(t)) |^2 $ but it is not clear how to close the estimates. Instead we notice that in dimension three
\begin{equation*} 
|\calS_{\eps}^{in}|^{1/3} \|  \ell_{\eps} - u_{\eps}(t, h_{\eps}(t)) \|_{L^2(0,T)}^2 \leq \int_{0}^{T} \int_{\bbR^d} |D(u_{\eps}- R_{\eps}(u))|^2
\end{equation*}
and we look for uniform bounds for the term $ |\calS_{\eps}^{in}|^{-1/6} (D(u_{\eps}- R_{\eps}(u))) $.

Let now introduce the definition of weak solutions and the main result.


\section{Definition of weak solutions and main result}

In this section we recall the concept of  finite energy weak solutions for the system \eqref{fluid:equ}-\eqref{body:equ}-\eqref{in:cond}. Then we state the main result of the paper. 

 Let introduce some notations from \cite{He:3D}. Let denote by 
\begin{equation*}
 \rho = \chi_{\calF(t)} + \rho_{\calS} \chi_{\calS(t)}
\end{equation*}
which is the extension by $ 1 $ of the density of the rigid body. Here for a set $ A \subset \bbR^{d} $, we denote by $ \chi_{A} $ the indicator function of $ A $, more precisely $ \chi_A(x) = 1 $ for $ x \in A $ and $ 0 $ elsewhere. Similarly we define the global velocity 
\begin{equation*}
u = u_{\F} \chi_{\calF(t)} + u_{\calS} \chi_{\calS(t)} = u \chi_{\calF(t)} + \left(\ell(t) + \omega \times (x-h(t)) \right) \chi_{\calS(t)}.
\end{equation*}
Notice that if $ u^{in} \in L^{2}(\calF(0))$, then the compatibility conditions \eqref{comp:cond} on the initial data imply that $ \DIV( u^{in}) = 0 $ in an appropriate weak sense. After all this preliminary we introduce the definition of Hopf-Leray type weak solutions for the system \eqref{fluid:equ}-\eqref{body:equ}-\eqref{in:cond}.

\begin{Definition}

Let $ \calS^{in} $ and $ \rho_{\calS}^{in} $ the initial position and density of the rigid body, let $ (u^{in}_{\calF}, \ell^{in}, \omega^{in}) $ satisfying the hypothesis \eqref{comp:cond} and such that $ u^{in } \in L^2(\bbR^d) $. Then a triple $ ( u_{\F}, \ell, \omega ) $ is a Hopf-Leray weak solution for the system\eqref{fluid:equ}-\eqref{body:equ}-\eqref{in:cond} associated with the initial data  $ \calS^{in} $, $ \rho^{in}_{\calS} $, $ u^{in}_{\F} $, $\ell^{in}$ and $ \omega^{in} $, if 
\begin{itemize}

\item the functions $ u_{\F} $, $ \ell $ and $ \omega $ satisfy 
\begin{gather*}
\ell \in L^{\infty}(\bbR^+;\bbR^d), \quad \omega \in L^{\infty}(\bbR^+; \bbR^{2d-3}) \\
u_{\F} \in L^{\infty }(\bbR^+;L^{2}(\calF(t))) \cap L^2_{loc}(\bbR^+; H^{1}(\calF(t))), \quad \text{and} \quad u \in C_w(\bbR^+; L^2(\bbR^d));
\end{gather*}

\item the vector field $ u $ is divergence free in $ \bbR^d $ with $ D(u) = 0 $ in $ \calS(t) $; 

\item the vector field $ u $ satisfies the equation in the following sense:
\begin{equation*}
-\int_{\bbR^+} \int_{\bbR^d} \rho u \cdot ( \partial_t \varphi + (u \cdot \nabla) \varphi ) - 2 \nu D(u): D(\varphi) \, dx \, dt = \int_{\bbR^d} \tilde{\rho}^{in} u^{in} \cdot \varphi(0,.) \, dt,
\end{equation*}
for any test function $ \varphi \in C^{\infty}_{c}(\bbR^+ \times \bbR^d ) $ such that $ \DIV(\varphi) = 0 $ and $ D(\varphi) = 0 $ in $ \calS(t )$. 

\item The following energy inequality holds
\begin{equation*}
\int_{\bbR^{d}} \rho(t,.) |u(t,.)|^2 \, dx + 4 \nu \int_{0}^t \int_{\bbR^d} |D(u)| ^2\, dx dt \leq \int_{\bbR^d} \rho|u^{in}|^2,
\end{equation*}
for almost any time $ t \in \bbR^+ $.

\end{itemize}

\end{Definition}

The existence of weak solutions for the system \eqref{fluid:equ}-\eqref{body:equ}-\eqref{in:cond} is now classical and can be found for example in \cite{Fei}.

\begin{Theorem}
For initial data  $ \calS^{in} $, $ \rho^{in}_{\calS} $, $ u^{in}_{\F} $, $\ell^{in}$ and $ \omega^{in} $ satisfying the hypothesis \eqref{comp:cond} and such that $ u^{in } \in L^2(\bbR^d) $, there exist a Hopf-Leray weak solution $ ( u_{\F}, \ell, \omega ) $ of the system \eqref{fluid:equ}-\eqref{body:equ}-\eqref{in:cond}.
\end{Theorem}

Let now introduce a small parameter $ \eps > 0 $ that control the size of the rigid body. We assume that $ \calS^{in}_{\eps} \subset B_{0}(\eps) $. We study the dynamics of the rigid body as $ \eps $ goes to zero for solutions of the system \eqref{fluid:equ}-\eqref{body:equ}-\eqref{in:cond} under some assumptions on the initial data $ \rho_{\calS, \eps} $, $ u^{in}_{\F, \eps} $, $\ell^{in}_{\eps} $ and $ \omega^{in}_{\eps} $ and in particular we show that in the limit the small rigid body follows the fluid. This result can be resumed as follows. 

\begin{Theorem}
\label{theo:main}
Let $ (u_{\F \eps}, \ell_{\eps}, \omega_{\eps}) $ a sequence of Hopf-Leray weak solutions associated with the initial data $ \calS^{in}_{\eps} $, $ \rho^{in}_{\calS,\eps} $, $ u^{in}_{\F, \eps} $, $\ell^{in}_{\eps} $ and $ \omega^{in}_{\eps} $  satisfying the hypothesis \eqref{comp:cond} and such that $ u^{in }_{\calF_{\eps}, \eps} \in L^{2}(\calF^{in})$ and let $ (u, p ) $ a classical solution to the Navier-Stokes equations in $[0,T] \times \bbR^d$ with initial data $ u^{in } \in H^{k}(\bbR^d) $ for $  k > d/2 + 1 $. If we assume that 
\begin{itemize}

\item the rigid body $ \calS^{in}_{\eps} \subset B_{0}(\eps) $;

\item $ m_{\eps} $ and $ \calS^{in}_{\eps} $ satisfy \eqref{ass};
 
\item for $ d = 3$, $ \| u_{\eps}^{in} - u^{in}\|_{L^{2}(\calF_{\eps})}/ |\calS_{\eps}^{in}|^{1/3} \lra 0 $ and $ m_{\eps}|\ell_{\eps}^{in}-u^{in}(h_{\eps}(0))| / |\calS_{\eps}^{in}|^{1/3} \lra 0 $;

\item  for $ d = 2$, $ \| u_{\eps}^{in} - u^{in}\|_{L^{2}(\calF_{\eps})}/ |\calS_{\eps}^{in}|^{\delta} \lra 0 $ and $ m_{\eps}|\ell_{\eps}^{in}-u^{in}(h_{\eps}(0))| / |\calS_{\eps}^{in}|^{\delta} \lra 0 $ where $ \delta > 0 $ is the same of condition \eqref{ass}. 
\end{itemize}  
Then up to subsequence 
\begin{equation*}
h_{\eps} \cv h \quad \text{ in } H^{1}(0,T) \quad \text{and} \quad  \ell_{\eps} \lra u(t, h(t)) \quad \text{ in } L^{2}(0,T).
\end{equation*}
Moreover
\begin{equation*}
h(t) = h(0) + \int_{0}^{t} u(t,h(t)).
\end{equation*}

\end{Theorem}

Let us notice that in dimension two the time of existence of regular solution is $ T = + \infty $. In this case the convergence of $ h_{\eps} $ and $ \ell_{\eps} $ holds in any compact interval. In dimension three the existence of global regular solutions is an open problem but there exist local in time solutions and they are global in time for small initial data. 

\section{Proof of the main result}

The prove of Theorem \ref{theo:main} is based on a relative energy inequality that it is stated in Lemma \ref{main:lemma}. We present the proof of Lemma \ref{main:lemma} in a separate section because it  is technical. The plan for this section is to recall the definition of restriction operator $ R_{\eps} $, to state Lemma \ref{main:lemma} and prove Theorem \ref{theo:main}. In the remaining part of the paper we set $ \nu = 1 $ to simplify the notation.  

\subsection{The restriction operator}

Let us introduce a restriction operator introduced in \cite{FRZ}. Consider a cut off $ \eta \in C^{\infty}_{c}(B_2(0)) $ such that $ 0 \leq \eta \leq 1 $ and $ \eta = 1 $ in $ B_1(0) $, moreover we introduce $ \eta_{\eps}(x) = \eta(x/ \eps) $. The restriction operator is defined as  
\begin{align*}
R_{\eps}[\varphi](t,x) = & \, \eta_{\eps}(x-h_{\eps}(t)) \varphi(t,x) + (1-\eta_{\eps}(x-h_{\eps}(t))) \varphi(t,h_{\eps}(t))  \\  & \, + B_{\eps}[\DIV(\eta_{\eps}(x-h_{\eps}(t)) \varphi(t,x) + (1-\eta_{\eps}(x-h_{\eps}(t))) \varphi(t,h_{\eps}(t)))],
\end{align*}
where $ B_{\eps}[f](x) =\eps  B_1[f(\eps y)](x/\eps) $ and $ B_1 $ is a Bogovski\u{\i} in $ B_{2}(0) \setminus B_{1}(0) $. See Appendix \ref{app:B} for more details.

To simplify the notation for any regular enough function $ \varphi $, we denote by by $ \bar{\varphi}^{\eps}(t) = \varphi(t, h_{\eps}(t))$. This allows us to rewrite the restriction operator in the more compact form  
\begin{align*}
R_{\eps}[\varphi] = & \, \eta_{\eps} \varphi + (1-\eta_{\eps}) \bar{\varphi}^{\eps}  + B_{\eps}[\DIV(\eta_{\eps} \varphi + (1-\eta_{\eps})\bar{\varphi}^{\eps})].
\end{align*}
In the following lemma we resume some properties of  $ R_{\eps} $.

\begin{Lemma}
The restriction operator $ R_{\eps} $ has the following properties. For any $ \varphi \in C^{0}(\bbR^d) $ such that $ \DIV(\varphi) =  0 $ and for $ p \in [1,+\infty] $,
\begin{gather*}
\DIV(R_{\eps}[\varphi]) = 0 \quad \text{ in } \bbR^{d}, \quad R_{\eps}[\varphi] = \varphi(t,h_{\eps}(t)) \quad \text{ in } \calS_{\eps}(t).
\end{gather*}
and
\begin{equation}
\label{ine:restr:lp}
\| R_{\eps}(\varphi) - \varphi  \|_{L^{p}(\bbR^{d})} \leq C \eps^{d/p}\|\varphi \|_{L^{\infty}(\bbR^d)} 
\end{equation}

Moreover if $  \varphi \in W^{1,\infty}_x(\bbR^d) $ such that $ \DIV(\varphi) =  0 $ and if  $ p \in (1,+\infty) $,
\begin{equation}
\label{ine:restr:grad}
\| \nabla R_{\eps}(\varphi) -\nabla  \varphi  \|_{L^{p}(\bbR^{d})} \leq C_p \eps^{d/p}\|\varphi \|_{W^{1,\infty}(\bbR^d)}.
\end{equation}

\end{Lemma}

\begin{proof}
Let us show estimates \eqref{ine:restr:lp}-\eqref{ine:restr:grad}. In the following let us use the notation $L^p_x$ to denote $ L^{p}(\bbR^d)$. We have
\begin{equation*}
\| R_{\eps}(\varphi) - \varphi  \|_{L^{p}_x} \leq \| R_{\eps}(\varphi) - \eta_{\eps} \varphi  \|_{L^{p}_x} + \| \eta_{\eps}\varphi  - \varphi  \|_{L^{p}_x} \leq \| R_{\eps}(\varphi) - \eta_{\eps} \varphi  \|_{L^{p}_x} + C\eps^{d/p}\| \varphi \|_{L^{\infty}_x}.
\end{equation*}
It is then enough to show 
\begin{equation}
\label{est:p} 
\| R_{\eps}(\varphi) - \eta_{\eps} \varphi  \|_{L^{p}_x} \leq C \eps^{d/p}\| \varphi \|_{L^{\infty}_x}.
\end{equation}
Using the definition of $ R_{\eps} $ and the fact that $ 1- \eta_{\eps} $ and $ B_{\eps} $ are supported in a ball of radius $ 2 \eps $, we have
\begin{align*}
\| R_{\eps}(\varphi) - \eta_{\eps} \varphi  \|_{L^{1}_x}  = & \, \|  (1-\eta_{\eps}) \bar{\varphi}^{\eps}  +  B_{\eps}[\DIV(\eta_{\eps} \varphi +  (1-\eta_{\eps})\bar{\varphi}^{\eps})] \|_{L^{1}_x} \\
\leq & \, C \eps^{d}\| \varphi \|_{L^{\infty}_x}  + C \eps \|  B_{\eps}[\DIV(\eta_{\eps} \varphi +  (1-\eta_{\eps})\bar{\varphi}^{\eps})] \|_{L^{n/(n-1)}_x} \\
\leq &  \, C \eps^{d}\| \varphi \|_{L^{\infty}_x}  + C \eps \|  \DIV(\eta_{\eps} \varphi +  (1-\eta_{\eps})\bar{\varphi}^{\eps}) \|_{L^{1}_x}
\end{align*}
Using the fact that $ \varphi $ is divergence free we can rewrite
\begin{equation*}
\DIV(\eta_{\eps} \varphi +  (1-\eta_{\eps})\bar{\varphi}^{\eps}) = \nabla \eta_{\eps}(\varphi- \bar{\varphi}^{\eps}).
\end{equation*}
This allows us to  deduce
\begin{align}
\| R_{\eps}(\varphi) - \eta_{\eps} \varphi  \|_{L^{1}_x}  \leq  &  \, C \eps^{d}\| \varphi \|_{L^{\infty}_x}  + C \eps \|   \nabla \eta_{\eps}(\varphi- \bar{\varphi}^{\eps}) \|_{L^{1}_x} \nonumber \\
\leq & \, C \eps^{d}\| \varphi \|_{L^{\infty}_x} + C \eps \|\nabla \eta_{\eps} \|_{L^{1}}\| \varphi \|_{L^{\infty}_x}  \nonumber \\
\leq & \, C \eps^{d}\| \varphi \|_{L^{\infty}_x}  . \label{est:1}
\end{align}
Similarly
\begin{align}
\| R_{\eps}(\varphi) - \eta_{\eps} \varphi  \|_{L^{\infty}_x}  = & \, \|  (1-\eta_{\eps}) \bar{\varphi}^{\eps}  +  B_{\eps}[\DIV(\eta_{\eps} \varphi +  (1-\eta_{\eps})\bar{\varphi}^{\eps})] \|_{L^{\infty}_x} \nonumber \\
\leq & \, C \| \varphi \|_{L^{\infty}_x}  + C \|  B_{\eps}[\DIV(\eta_{\eps} \varphi +  (1-\eta_{\eps})\bar{\varphi}^{\eps})] \|_{L^{\infty}_x} \nonumber  \\
\leq &  \, C \| \varphi \|_{L^{\infty}_x}  + C  \|  \DIV(\eta_{\eps} \varphi +  (1-\eta_{\eps})\bar{\varphi}^{\eps}) \|_{L^{d}_x}\nonumber  \\
\leq & \, C \| \varphi \|_{L^{\infty}_x}  +  C \|\nabla \eta_{\eps} \|_{L^{d}}\| \varphi \|_{L^{\infty}_x} \nonumber \\
\leq & \, C \| \varphi \|_{L^{\infty}_x}. \label{est:infty}
\end{align}
From \eqref{est:1}-\eqref{est:infty} and interpolation inequality, we deduce \eqref{est:p}.

To show estimate \eqref{ine:restr:grad}, it is enough to show 
\begin{equation}
\label{est:p:nabla} 
\| \nabla R_{\eps}(\varphi) - \eta_{\eps} \nabla \varphi  \|_{L^{p}_x} \leq C_p \eps^{d/p}\| \varphi \|_{L^{\infty}_x}.
\end{equation}
Notice that 
\begin{align*}
 \nabla R_{\eps}(\varphi) - \eta_{\eps} \nabla \varphi = \nabla \eta_{\eps}(\varphi - \bar{\varphi}^{\eps}) +  \nabla B_{\eps}[\nabla \eta_{\eps}(\varphi - \bar{\varphi}^{\eps}) ].
\end{align*}
and 
\begin{align*}
\| \eta_{\eps}(\varphi - \bar{\varphi}^{\eps})  \|_{L^{p}_x} \leq & \,  \left\|  |.-h_{\eps}(t)| \eta_{\eps}(.)\frac{\varphi(.) - \bar{\varphi}^{\eps}}{|.-h_{\eps}(t)|}    \right\|_{L^{p}_x} \\
\leq & \|  |.-h_{\eps}(t)| \eta_{\eps}(.) \|_{L^{p}_x} \left\|   \frac{\varphi(.) - \bar{\varphi}^{\eps}}{|.-h_{\eps}(t)|}    \right\|_{L^{\infty}_x} \\
\leq & \, C \eps^{d/p} \| \varphi \|_{W^{1,\infty}_x}
\end{align*}
The above computation and estimate allow us to deduce
\begin{align*}
 \| \nabla R_{\eps}(\varphi) - \eta_{\eps} \nabla \varphi  \|_{L^{p}_x} \leq & \, \| \nabla \eta_{\eps}(\varphi - \bar{\varphi}^{\eps}) \|_{L^{p}_x} + \| \nabla B_{\eps}[\nabla \eta_{\eps}(\varphi - \bar{\varphi}^{\eps}) ]\|_{L^{p}_x} \\
 \leq & \, C_p \| \nabla \eta_{\eps}(\varphi - \bar{\varphi}^{\eps}) \|_{L^{p}_x} \\
 \leq & \, C_p \eps^{d/p} \| \varphi \|_{W^{1,\infty}_x}.
\end{align*}

\end{proof}

\begin{Remark}  
Let us notice that in the proof of the above lemma we show also inequality \eqref{est:p} that reads
\begin{equation}
\label{est:lp:2}
\| R_{\eps}(\varphi) - \eta_{\eps} \varphi  \|_{L^{p}(\bbR^{d})} \leq C \eps^{d/p}\|\varphi \|_{L^{\infty}_x} 
\end{equation}
for any $ \varphi \in C^{0}(\bbR^d) $ such that $ \DIV(\varphi) =  0 $ and for $ p \in [1,+\infty] $ and estimate  \eqref{est:p:nabla} that reads
\begin{equation}
\label{est:nabla:2}
\| \nabla R_{\eps}(\varphi) -\eta_{\eps} \nabla  \varphi  \|_{L^{p}(\bbR^{d})} \leq C \eps^{d/p}\|\varphi \|_{W^{1,\infty}_x}.
\end{equation}
 for any $  \varphi \in W^{1,\infty}_x(\bbR^d) $ such that $ \DIV(\varphi) =  0 $ and if  $ p \in (1,+\infty) $.

\end{Remark}

We will now present the relative energy inequality.

\subsection{Relative energy inequality}

We will use the above restriction operator to deduce a relative energy inequality.

\begin{Lemma}
\label{main:lemma}
Under the hypothesis of Theorem \ref{theo:main}, we have
\begin{align*}
\int_{\bbR^d} \rho_{\eps}|u_{\eps}(t,.) - R_{\eps}(u(t,.))|^2 + 4 & \int_{0}^{T} \int_{\bbR^d} |D(u_{\eps}- R_{\eps}(u))|^2  \leq \int_{\bbR^d} \rho_{\eps}|u^{in}_{\eps} - R_{\eps}(u^{in})|^2 \\ & +  C \int_0^T \int_{\bbR^d} \rho_{\eps}|u_{\eps} - R_{\eps}(u)|^2+ Rest_{\eps},
\end{align*}
where 
$$  |Rest_{\eps}| \leq C ( m_{\eps} + |\calS_{\eps}|+ \eps^{3/2} + \eps^{3}|\calS^{in}_{\eps}|^{-1/3}), \quad \text{ for } d = 3, $$
$$  |Rest_{\eps}| \leq C ( m_{\eps} + |\calS_{\eps}|+ \eps^{\tilde{\delta}} + \eps^{1+ \tilde{\delta}} |\calS^{in}_{\eps}|^{-\delta}), \quad \text{ for } d = 2, $$
with $ C $ independent of $ \eps $.

\end{Lemma}

The above lemma is the key estimate to deduce Theorem \ref{theo:main}. 

\subsection{Proof of Theorem \ref{theo:main}}

Let us show Theorem  \ref{theo:main} with the help of Lemma \ref{main:lemma}.

\begin{proof}[Proof of Theorem \ref{theo:main}]
Let us show Theorem \ref{theo:main} in dimension  three and explain at the end how to adapted in dimension two. First of all let notice that
\begin{align*}
|\calS_{\eps}^{in}|^{1/2} \| \ell_{\eps} - u(t, h_{\eps}(t)) \|_{L^2_t} \leq & \, \| u_{\eps} - R_{\eps}(u)\|_{L^{2}_t(L^2(\calS_{\eps}(t)))} \\
\leq & \,  |\calS_{\eps}^{in}|^{1/3}  \| u_{\eps} - R_{\eps}(u)\|_{L^{2}_t(L^6(\calS_{\eps}(t)))}  \\
\leq & \,  C |\calS_{\eps}^{in}|^{1/3}  \| D (u_{\eps} - R_{\eps}(u))\|_{L^{2}_t(L^2(\bbR^3 ))},
\end{align*}
where we used that $ \rho_{\eps} $ is constant in $ \calS_{\eps} $ for any fixed $ \eps $ to deduce the first inequality. We can rewrite the above inequality as 
\begin{equation*}
|\calS_{\eps}^{in}|^{1/3} \| \ell_{\eps} - u(t, h_{\eps}(t)) \|_{L^2_t}^2 \leq  C \| D (u_{\eps} - R_{\eps}(u))\|_{L^{2}_t(L^2(\bbR^3(t)))}^2.
\end{equation*}    
From Lemma \ref{main:lemma}, 
\begin{align*}
|\calS_{\eps}^{in}|^{-1/3}& \int_{\bbR^3} \rho_{\eps}|u_{\eps}(t,.) - R_{\eps}(u(t,.))|^2 + C^{-1} \| \ell_{\eps} - u(t, h_{\eps}(t)) \|_{L^2_t}^2 \\
 \leq & \, |\calS_{\eps}^{in}|^{-1/3}\int_{\bbR^3} \rho_{\eps}|u_{\eps}(t,.) - R_{\eps}(u(t,.))|^2  + 4 |\calS_{\eps}^{in}|^{-1/3} \int_{0}^{T} \int_{\bbR^3} |D(u_{\eps}- R_{\eps}(u))|^2  \\
 \leq & \, |\calS_{\eps}^{in}|^{-1/3} \int_{\bbR^3} \rho_{\eps}|u^{in}_{\eps} - R_{\eps}(u^{in})|^2 +  C |\calS_{\eps}^{in}|^{-1/3} \int_0^T \int_{\bbR^3} \rho_{\eps}|u_{\eps} - R_{\eps}(u)|^2 \\
 & \, + |\calS_{\eps}^{in}|^{-1/3} Rest_{\eps}. 
\end{align*}
Gr\"omwall's inequality implies that
\begin{align}
\label{main:grom:est}
\| \ell_{\eps} - u(t, h_{\eps}(t)) \|_{L^2_t}^2 \leq \left(  |\calS_{\eps}^{in}|^{-1/3} \int_{\bbR^3} \rho_{\eps}|u^{in}_{\eps} - R_{\eps}(u^{in})|^2  + |\calS_{\eps}^{in}|^{-1/3} |Rest_{\eps}|\right) e^{TC}.
\end{align}
Notice that 
\begin{align*}
\int_{\calF_{\eps}(t) } |u^{in}_{\eps} - R_{\eps}(u^{in})|^2 \leq & \, \int_{\calF_{\eps}(t) } |u^{in}_{\eps} - u^{in}|^2 - 2 \int_{\calF_{\eps}(t) }u_{\eps}^{in}\cdot (u^{in} - R_{\eps}(u^{in})) \\
& \,  + \int_{\calF_{\eps}(t) } (u^{in} - R_{\eps}(u^{in})) \cdot R_{\eps}(u^{in}) + \int_{\calF_{\eps}(t) } u^{in}\cdot (u^{in} - R_{\eps}(u^{in})) \\
\leq  & \, \int_{\calF_{\eps}(t) } |u^{in}_{\eps} - u^{in}|^2 + C \eps^{3/2},
\end{align*}
where we used some H\"older inequalities and \eqref{ine:restr:lp} for $ d = 3 $, $ p = 2 $ and $ \varphi = u^{in} $. 
%
%
%
%

By hypothesis we have 

\begin{align}
\label{const:to:zero}
\frac{1}{|\calS_{\eps}^{in}|^{1/3} } \int_{\bbR^3} \rho_{\eps}|u^{in}_{\eps} - R_{\eps}(u^{in})|^2 = & \,   \frac{1}{|\calS_{\eps}^{in}|^{1/3} } \int_{\calF_{\eps}(t)} |u^{in}_{\eps} - u^{in}|^2  + \frac{m_{\eps}}{|\calS_{\eps}^{in}|^{1/3}}|\ell_{\eps} - u^{in}(h_{\eps}(0))|^2 
\nonumber \\ & \, + C \frac{\eps^{3/2}}{|\calS_{\eps}^{in}|^{1/3}}  \lra 0
\end{align}
and by Lemma \ref{main:lemma}
\begin{equation}
\label{LLL}
|\calS_{\eps}^{in}|^{-1/3}| Rest_{\eps}| \leq  |\calS_{\eps}^{in}|^{-1/3}| C ( m_{\eps} + |\calS_{\eps}|+ \eps^{3/2} + \eps^{3}|\calS^{in}_{\eps}|^{-1/3}) \lra 0
\end{equation}
from hypothesis 
\begin{equation*}
\frac{m_{\eps}}{|\calS_{\eps}^{in}|^{1/3}} \lra 0 \quad \text{ and } \quad \frac{|\calS^{in}_{\eps}|}{\eps^{9/2}} \lra + \infty.
\end{equation*}

The estimate \eqref{main:grom:est} together with \eqref{const:to:zero} and \eqref{LLL} ensure that 
$$ \|\ell_{\eps}\|_{L^2_t} \leq C, $$
which implies that $ h_\eps $ is uniformly bounded in $ W^{1,2}_t $. Up to subsequence 
\begin{equation*} 
h_{\eps} \cv h \text{ in } W^{1,2},  
\end{equation*}
in particular it converge strong in $ C^0_t $. We deduce that \eqref{main:grom:est} together with \eqref{const:to:zero} imply 

\begin{equation*}
\ell_{\eps} \to u(t, h(t)) \quad \text{ in } L^2_t.
\end{equation*}

Finally we pass to the limit in the equation 
\begin{equation*}
h_{\eps}(t) = h_{\eps}(0) + \int_{0}^{t}\ell_{\eps}
\end{equation*}
to deduce
\begin{equation*}
h(t) = h(0) + \int_{0}^{t} u(t,h(t)).
\end{equation*}

Let now move to the case of dimension two.
First of all let notice that
\begin{align*}
|\calS_{\eps}^{in}|^{1/2} \| \ell_{\eps}& - u(t, h_{\eps}(t)) \|_{L^2_t} \leq   \| u_{\eps} - R_{\eps}(u)\|_{L^{2}_t(L^2(\calS_{\eps}(t)))} \\
\leq & \,  |\calS_{\eps}^{in}|^{1/q}  \| u_{\eps} - R_{\eps}(u)\|_{L^{2}_t(L^p(\calS_{\eps}(t)))}  \\
\leq & \,  C_p |\calS_{\eps}^{in}|^{1/q} (  \| u_{\eps} - R_{\eps}(u)\|_{L^{2}_t(L^2(\calF_{\eps}(t)))} +   \| D (u_{\eps} - R_{\eps}(u))\|_{L^{2}_t(L^2(\bbR^2 ))}),
\end{align*}
where $ 1/p + 1/ q = 1/ 2 $ and we used Lemma \ref{lemma:sob:emb} for the last inequality. We deduce that 
\begin{align*}
|\calS_{\eps}^{in}|^{1/p} \| \ell_{\eps}& - u(t, h_{\eps}(t)) \|_{L^2_t} \leq 
 C_p  (  \| u_{\eps} - R_{\eps}(u)\|_{L^{2}_t(L^2(\calF_{\eps}(t)))} +   \| D (u_{\eps} - R_{\eps}(u))\|_{L^{2}_t(L^2(\bbR^2 ))}),
\end{align*}
for any $ p < \infty $.

Choose now $ 1/p = \delta $. Using Lemma \ref{main:lemma} we deduce that 
\begin{align*}
| \calS_{\eps}^{in }|^{- \delta } & \int_{\bbR^d} \rho_{\eps}|u_{\eps}(t,.) - R_{\eps}(u(t,.))|^2  + \frac{4}{C_P} \| \ell_{\eps} - u(t, h_{\eps}(t)) \|_{L^2_t} \\ \leq & \, 
| \calS_{\eps}^{in }|^{- \delta } \int_{\bbR^d} \rho_{\eps}|u_{\eps}(t,.) - R_{\eps}(u(t,.))|^2 
 + 4 | \calS_{\eps}^{in }|^{- \delta } \int_{0}^{T} \int_{\bbR^d} |D(u_{\eps}- R_{\eps}(u))|^2 \\
& \, + 4  \int_{0}^{T} \int_{\calF_{\eps}(t)} |u_{\eps}- R_{\eps}(u)|^2 \\
\leq & \, | \calS_{\eps}^{in }|^{- \delta } \int_{\bbR^d} \rho_{\eps}|u^{in}_{\eps} - R_{\eps}(u^{in})|^2 + \left(C + 4 \right) | \calS_{\eps}^{in }|^{- \delta }\int_0^T \int_{\bbR^d} \rho_{\eps}|u_{\eps} - R_{\eps}(u)|^2 \\
& \, + | \calS_{\eps}^{in }|^{- \delta } Rest_{\eps}.
\end{align*}
Gr\"omwall's inequality implies that
\begin{align*}
\| \ell_{\eps} - u(t, h_{\eps}(t)) \|_{L^2_t}^2 \leq \left(  |\calS_{\eps}^{in}|^{-1/3} \int_{\bbR^3} \rho_{\eps}|u^{in}_{\eps} - R_{\eps}(u^{in})|^2  + |\calS_{\eps}^{in}|^{-\delta} |Rest_{\eps}|\right) e^{TC}.
\end{align*}
Using the assumptions \eqref{ass} and following the same strategy as in the case of dimension three, we prove the theorem. 

\end{proof}

\section{Proof of Lemma \ref{main:lemma}}

Let us now show Lemma \ref{main:lemma}.

\begin{proof} [Proof of Lemma \ref{main:lemma}] By the hypothesis of Theorem \ref{theo:main} that $ u^{in} \in H^{k} $ for $ k > d/2 + 1 $, there exists a local solution in dimension three and a global solution in dimension two of the Navier-Stokes equations such that
$$ u \in L^{\infty}(0,T;H^k(\bbR^d)) \cap L^{2}(0,T;H^{k+1}(\bbR^d)).  $$  
With this choice of  $ k $ we have also the bounds
\begin{equation*}
\|\partial_t u \|_{L^{2}_{t}(L^{\infty}_x) } + \| u \|_{L^{\infty}_{t}(W^{1,\infty}_{x})} + \| \nabla p \|_{ L^{2}_{t}(L^{\infty}_x)  } \leq C.
\end{equation*}
The above bound will be implicitly used in many of the estimates to prove this lemma.  

To deduce the relative energy inequality, let start by computing  
\begin{align*}
 \int_{\bbR^d} & \rho_{\eps} |u_{\eps}(t,.)- R_{\eps}(u(t,.))|^2 = \int_{\bbR^d} \rho_{\eps} |u_{\eps}(t,.)|^2 - 2 \int_{\bbR^d} \rho_{\eps} u_{\eps}(t,.)\cdot R_{\eps}(u(t,.)) \\
  & \, +  \int_{\bbR^d} \rho_{\eps} | R_{\eps}(u(t,.))|^2 \\
 \leq  & \,  \int_{\bbR^d} \rho_{\eps} |u_{\eps}^{in}|^2 - 4 \int_{0}^{T}  \int_{\bbR^d} |D u_{\eps}|^2 -2 \int_{0}^T  \int_{\bbR^d} \rho_{\eps} u_{\eps}(t,.)\cdot \partial_t R_{\eps}(u) \\ 
 & \, - 2 \int_0^T  \int_{\bbR^d} u_{\eps} \otimes u_{\eps}: \nabla R_{\eps}(u) + 4 \int_0^T \int_{R^d} Du_{\eps}: DR_{\eps}(u) + 2 \int_{\bbR^d} \rho^{in}u_{\eps}^{in} \cdot R_{\eps}(u) \\
& + \int_{\bbR^d} \rho^{in}|R_{\eps}(u^{in})|^2 - 4  \int_0^T \int_{\bbR^d} |DR_{\eps}(u)|^2 +  \int_{\bbR^d} \rho_{\eps} | R_{\eps}(u(t,.))|^2 -  \int_{\bbR^d} |u(t,.)|^2  \\ 
& \, +  \int_{\bbR^d} \rho_{\eps} |u^{in}|^2 -  \int_{\bbR^d} \rho_{\eps}^{in}|R_{\eps}(u^{in})|^2 + \int_0^{T}\int_{\bbR^d} |DR_{\eps}(u)|^2 - 4 \int_{\bbR^d}|Du|^2,
\end{align*}
where in the inequality we use the energy inequality for $ u_{\eps }$, the weak formulation satisfied by $ u_{\eps} $ and the energy inequality satisfied by $ u $. 
After bringing on the left hand side some terms involving $ Du_{\eps} $  and $ DR_{\eps}(u )$ , we deduce 
\begin{align*}
\int_{\bbR^d} \rho_{\eps}|u_{\eps}(t,.) - R_{\eps}(u(t,.))|^2 + 4 & \int_{0}^{T} \int_{\bbR^d} |D(u_{\eps}- R_{\eps}(u))|^2  \\ & \leq \int_{\bbR^d} \rho_{\eps}|u^{in}_{\eps} - R_{\eps}(u^{in})|^2 + \widetilde{Rest_{\eps}}
\end{align*}
where
\begin{align*}
\widetilde{Rest_{\eps}} = & \,  -2 \int_{0}^T  \int_{\bbR^d} \rho_{\eps} u_{\eps}(t,.)\cdot \partial_t R_{\eps}(u) - 2 \int_0^T  \int_{\bbR^d} u_{\eps} \otimes u_{\eps}: \nabla R_{\eps}(u)  \\
& \, - 4 \int_0^T \int_{\bbR^d} Du_{\eps}: DR_{\eps} (u) +  \int_{\bbR^d} \rho_{\eps} | R_{\eps}(u(t,.))|^2 -  \int_{\bbR^d} |u(t,.)|^2  \\
& \, +  \int_{\bbR^d} \rho_{\eps} |u^{in}|^2 -  \int_{\bbR^d} \rho_{\eps}^{in}|R_{\eps}(u^{in})|^2 +4 \int_0^{T}\int_{\bbR^d} |DR_{\eps}(u)|^2 - 4 \int_{\bbR^d}|Du|^2.
\end{align*}
It remains to estimate $ |\widetilde{Rest_{\eps}}|$. To do that, we decompose the remainder $ \widetilde{Rest_{\eps}} = Rest_{\eps}^{1} + Rest_{\eps}^{2} $ where 
\begin{align*}
Rest_{\eps}^{2} =  & \, \int_{\bbR^d} \rho_{\eps} | R_{\eps}(u(t,.))|^2 -  \int_{\bbR^d} |u(t,.)|^2  +   \int_{\bbR^d} \rho_{\eps} |u^{in}|^2 -  \int_{\bbR^d} \rho_{\eps}^{in}|R_{\eps}(u^{in})|^2 \\
& \, + 4 \int_0^{T}\int_{\bbR^d} |DR_{\eps}(u)|^2 - 4 \int_{\bbR^d}|Du|^2.
\end{align*}
We start by estimating the terms 
\begin{align*}
\left| \int_{\bbR^d} \rho_{\eps} | R_{\eps}(u(t,.))|^2 -  \int_{\bbR^d} |u(t,.)|^2 \right| \leq & \,  (m_{\eps} + |\calS_{\eps}^{in}|)\|u\|_{L^{\infty}} \\ & + \left| \int_{\calF_{\eps}(t)} | R_{\eps}(u(t,.))|^2 -  \int_{\calF_{\eps}(t)} |u(t,.)|^2 \right|.
\end{align*}
To tackle the last term on the right hand side, 
%
%
we notice that $ R_{\eps}(u(t,.)) - u(t,.) $ is supported in $ B_{2 \eps}(h_{\eps}(t)) $ and that 
\begin{equation*}
\|R_{\eps}u \|_{L^{\infty}_x} \leq C\|u\|_{L^{\infty}_x}
\end{equation*}
from \eqref{ine:restr:lp}.
%
%
 %
 The two above observations allow us to estimate
 \begin{align*}
 \left| \int_{\calF_{\eps}(t)} | R_{\eps}(u(t,.))|^2 -  \int_{\calF_{\eps}(t)} |u(t,.)|^2 \right| \leq & \,  \left| \int_{\calF_{\eps}(t)} \rho_{\eps}  R_{\eps}(u(t,.))(R_{\eps}(u(t,.)) - u(t,.) )\right| \\
 & +  \left| \int_{\calF_{\eps}(t)} (R_{\eps}(u(t,.)) - u(t,.) )u(t,.)\right|  \\
 \leq & \, C\eps^d \|u\|_{L^{\infty}_x}^2.
\end{align*}
We deduce that
\begin{equation}
 \label{repubblica:1}
\left| \int_{\bbR^d} \rho_{\eps} | R_{\eps}(u(t,.))|^2 -  \int_{\bbR^d} |u(t,.)|^2 \right| \leq C ( m_{\eps} + |\calS_ {\eps}^{in}| + \eps^d ).
\end{equation}
We estimate the third and fourth terms of $ Rest^2 $ analogously and we deduce 
\begin{equation}
 \label{repubblica:2}
\left|\int_{\bbR^d} |u^{in}|^2 - \int_{\bbR^d} \rho_{\eps} | R_{\eps}(u^{in})|^2   \right| \leq C ( m_{\eps} + |\calS_ {\eps}^{in}| + \eps^d ).
\end{equation}
We are left with the estimate of
\begin{align*}
\left|  \int_0^{T}\int_{\bbR^d} |DR_{\eps}(u)|^2 -  \int_{\bbR^d}|Du|^2 \right|  \leq & \, \left|  \int_0^{T}\int_{\bbR^d} DR_{\eps}(u) :D(R_{\eps}(u) - u )\right| \\ & \, + \left| \int_{\bbR^d} D(R_{\eps}(u) - u ):  Du \right|
\end{align*}

From \eqref{ine:restr:grad}, we have 
\begin{align*}
\| DR_{\eps}(u) - D(u)\|_{L^{2}_t(L^2_t)} \leq & \, C(\| u \|_{L^2_t(W^{1,\infty}_x)}) \eps^{d/2}.
\end{align*}
which also implies 
\begin{align*}
\| DR_{\eps}(u) \|_{L^{2}_t(L^2_x)} \leq \| u \|_{L^{2}_t(H^1_x)} + \| u \|_{L^2_t(W^{1,\infty}_x)}.
\end{align*}
%

%
%
%
%
%
We deduce that 
\begin{align}
\nonumber
\left|  \int_0^{T}\int_{\bbR^d} |DR_{\eps}(u)|^2 -  \int_{\bbR^d}|Du|^2 \right|  \leq & \, (\| DR_{\eps}(u)\|_{L^{2}_t(L^2_x)} + \| Du\|_{L^{2}_t(L^2_x)})(\| D(R_{\eps}(u)-u)\|_{L^{2}_t(L^2_x)}) \\
\leq & \, C \eps^{d/2}. \label{repubblica:3} 
\end{align}
Collecting \eqref{repubblica:1}-\eqref{repubblica:2}-\eqref{repubblica:3}, we have
\begin{equation}
\label{Rest2:est}
|Rest^2_{\eps}| \leq C( m_{\eps} + \eps^{d/2} + |\calS_{\eps}^{in}|).
\end{equation}
We now consider the more difficult term $ Rest_{\eps}^1 $, which reads
 \begin{align*}
Rest_{\eps}^1 = & \,  -2 \int_{0}^T  \int_{\bbR^d} \rho_{\eps} u_{\eps}(t,.)\cdot \partial_t R_{\eps}(u) - 2 \int_0^T  \int_{\bbR^d} u_{\eps} \otimes u_{\eps}: \nabla R_{\eps}(u)  \\
& \, - 4 \int_0^T \int_{R^d} Du_{\eps}: DR_{\eps} (u).
\end{align*}
To tackle this term we compute the time derivative of $ R_{\eps}(u) $ and use the equation satisfied by $ u $.
\begin{align*}
\partial_t R_{\eps}[u] = & \, - \ell_{\eps} \cdot \nabla \eta_{\eps} u + \eta_{\eps} \partial_t u + \ell_{\eps} \cdot \nabla \eta_{\eps} \bar{u}^{\eps} + (1-\eta_{\eps}) \partial_t \bar{u}^{\eps} +  (1-\eta_{\eps}) \ell_{\eps} \cdot \nabla \bar{u}^{\eps}  \\ 
& \, - \ell_{\eps} \cdot \nabla B_{\eps} [ \nabla \eta_{\eps}(u-\bar{u}^{\eps})] + B_{\eps}[ \nabla \eta_{\eps}(\partial_t u - \partial_t \bar{u}^{\eps} + \ell_{\eps}\cdot \nabla u  - \ell_{\eps}\cdot \nabla \bar{u}^{\eps}))] .
\end{align*}
Let us notice that in the above expression there is not a time derivative of $ \eta_{\eps} $ inside $ B_{\eps} $ because the $ B_{\eps } $ follow the rigid body as $ \eta_{\eps} $. 
Let rewrite
\begin{equation*}
\partial_t R_{\eps}[u]  = \sum_{i=1}^{5} I_i,
\end{equation*}
where
\begin{gather*}
I_1 = \eta_{\eps} \partial_t u , \quad I_{2} =  - \ell_{\eps} \cdot \nabla \eta_{\eps} u + \ell_{\eps} \cdot \nabla \eta_{\eps} \bar{u}^{\eps} , \quad I_{3 } = (1-\eta_{\eps}) \partial_t \bar{u}^{\eps} +  (1-\eta_{\eps}) \ell_{\eps} \cdot \nabla \bar{u}^{\eps}, \\	
I_{4} =   - \ell_{\eps} \cdot \nabla B_{\eps} [ \nabla \eta_{\eps}(u-\bar{u}^{\eps})], \quad \text{ and } \quad I_5 =  B_{\eps}[ \nabla \eta_{\eps}(\partial_t u - \partial_t \bar{u}^{\eps} + \ell_{\eps}\cdot \nabla u  - \ell_{\eps}\cdot \nabla \bar{u}^{\eps}))] .
\end{gather*}
To tackle the term $ I_1 $, we use the equation satisfied by $ u $. So let start by bounding the other terms. From now on the estimates depend on the dimension so let us focus on the case of dimension three and in the end we explain how to adapted in dimension two. 
We have that
\begin{align}
\left| \int_{0}^T \int_{\bbR^d} \rho_{\eps} u_{\eps} \cdot I_2 \right| = & \, \left| \int_{0}^T \int_{\calF(t)} u_{\eps} \cdot (- \ell_{\eps} \cdot \nabla \eta_{\eps} u + \ell_{\eps} \cdot \nabla \eta_{\eps} \bar{u}^{\eps})   \right| \nonumber \\
\leq & \, \| u_{\eps}\|_{L^2_t(L^6_x)} \|\ell_{\eps}\|_{L^{2}_t} \|(x-h_{\eps}) \nabla \eta_{\eps}\|_{L^{\infty}_x} \| u\|_{L^{\infty}_t(W^{1,\infty}_x)} \eps^{5/2}  \nonumber \\
\leq & \, C \eps^{5/2} |\calS_{\eps}^{in}|^{-1/6}.   	\label{I2}
\end{align}
The term $ I_3 $ is the only one which is not zero in $ \calS_{\eps} $. We have 
\begin{align}
\left| \int_{0}^T \int_{\bbR^d} \rho_{\eps} u_{\eps} \cdot I_3 \right| = & \, \left| m_{\eps} \int_0^T \ell_{\eps} \cdot \partial_t \bar{u}^{\eps} + \ell_{\eps} \cdot \nabla  \bar{u}^{\eps}   + \int_{0}^T \int_{\calF_{\eps}(t)}  u_{\eps} \cdot  (1-\eta_{\eps}) (\partial_t \bar{u}^{\eps} + \ell_{\eps} \cdot \nabla \bar{u}^{\eps} ) \right|	  \label{I3:0}
\end{align}
To tackle the right hand side, we notice that   
\begin{align}
\left| m_{\eps} \int_0^T \ell_{\eps} \cdot \partial_t \bar{u}^{\eps} \right| \leq & \,  \left|  m_{\eps} \int_0^T (\ell_{\eps}-\bar{u}^{\eps}) \cdot \partial_t \bar{u}^{\eps} +  m_{\eps} \int_0^T \bar{u}^{\eps} \cdot \partial_t \bar{u}^{\eps}\right| \nonumber \\
\leq & \frac{1}{2} \int_{0}^T \int_{\bbR^d} \rho_{\eps}|u_{\eps}- R_{\eps}(u)|^2 +  m_{\eps} \| \partial_t u\|_{L^2_t(L^{\infty}_x)}( \| \partial_t u\|_{L^2_t(L^{\infty}_x)} + \| u\|_{L^2_t(L^{\infty}_x)} )  \label{I3:1}
\end{align}
and similarly
\begin{align}
\left| m_{\eps} \int_0^T \ell_{\eps} \cdot \nabla \bar{u}^{\eps} \right| \leq & \, \frac{1}{2} \int_{0}^T \int_{\bbR^d} \rho_{\eps}|u_{\eps}- R_{\eps}(u)|^2 + m_{\eps}(\| \nabla u\|_{L^{\infty}_x} + \| u\|_{L^2_t(L^{\infty}_x)}).  \label{I3:2}
\end{align}
Moreover
\begin{align}
& \left|  \int_{0}^T \int_{\calF_{\eps}(t)}  u_{\eps} \cdot  (1-\eta_{\eps}) (\partial_t \bar{u}^{\eps} + \ell_{\eps} \cdot \nabla \bar{u}^{\eps} ) \right|	 \nonumber \\
& \quad \quad \quad \quad \quad \quad   \leq  \|u_{\eps}\|_{L^2_t(L^6_x)} \left( \|\partial_t u \|_{L^2_t(L^{\infty}_x)}  
+ \|\ell_{\eps}\|_{L^2_t} \|\nabla u \|_{L^{\infty}_t(L^{\infty}_x)} \right) \eps^{5/2}.  \label{I3:3}
\end{align}
Equality \eqref{I3:0} and estimates \eqref{I3:1}-\eqref{I3:2}-\eqref{I3:3}, implies that  
\begin{align}
\left| \int_{0}^T \int_{\bbR^d} \rho_{\eps} u_{\eps} \cdot I_3 \right|  
\leq \int_{0}^T \int_{\bbR^d} \rho_{\eps}|u_{\eps}- R_{\eps}(u)|^2 + C \left( m_{\eps} + \eps^{5/2} + \eps^{5/2} |\calS_{\eps}^{in}|^{-1/6} \right). \label{I3}
\end{align}

For the term $ I_4 $ we proceed as follows.
\begin{align}
\left| \int_{0}^T \int_{\bbR^d} \rho_{\eps} u_{\eps} \cdot I_4 \right| = & \, \left| \int_{0}^T \int_{\calF(t)} u_{\eps} \cdot \ell_{\eps} \cdot \nabla B_{\eps} [ \nabla \eta_{\eps}(u-\bar{u}^{\eps})] \right| \nonumber \\
\leq & \, \| u_{\eps}\|_{L^2_t(L^6_x)} \|\ell_{\eps}\|_{L^{2}_t} \| \nabla B_{\eps}[ \nabla \eta_{\eps}(u-\bar{u}^{\eps})] \|_{L^{\infty}_t(L^{6/5}_x)} \nonumber \\
\leq & \, C \| u_{\eps}\|_{L^2_t(L^6_x)} \|\ell_{\eps}\|_{L^{2}_t} \| \nabla \eta_{\eps}(u-\bar{u}^{\eps}) \|_{L^{\infty}_t(L^{6/5}_x)} \nonumber \\
\leq & \, C  \| u_{\eps}\|_{L^2_t(L^6_x)} \|\ell_{\eps}\|_{L^{2}_t} \|(x-h_{\eps}) \nabla \eta_{\eps}\|_{L^{\infty}_x} \| u\|_{L^{\infty}_t(W^{1,\infty}_x)} \eps^{5/2}  \nonumber \\
\leq & \, C \eps^{5/2} |\calS_{\eps}^{in}|^{-1/6}.   	\label{I4}
\end{align}
Let now tackle $ I_5 $. 
\begin{align}
\left| \int_{0}^T \int_{\bbR^d} \rho_{\eps} u_{\eps} \cdot I_5 \right| = & \, \left| \int_{0}^T \int_{\calF(t)} u_{\eps} \cdot B_{\eps}[ \nabla \eta_{\eps}(\partial_t u - \partial_t \bar{u}^{\eps} + \ell_{\eps}\cdot \nabla u  - \ell_{\eps}\cdot \nabla \bar{u}^{\eps}))]  \right| \nonumber \\
\leq & \, C \| u_{\eps}\|_{L^2_t(L^6_x)}  \| B_{\eps}[ \nabla \eta_{\eps}(\partial_t u - \partial_t \bar{u}^{\eps} + \ell_{\eps}\cdot \nabla u  - \ell_{\eps}\cdot \nabla \bar{u}^{\eps}))]   \|_{L^{2}_t(L^2_x)} \eps  \nonumber \\
\leq & \, C \| u_{\eps}\|_{L^2_t(L^6_x)}  \| \nabla \eta_{\eps}(\partial_t u - \partial_t \bar{u}^{\eps} + \ell_{\eps}\cdot \nabla u  - \ell_{\eps}\cdot \nabla \bar{u}^{\eps}))   \|_{L^{2}_t(L^{6/5}_x)} \eps  \nonumber \\
\leq & \, C \| u_{\eps}\|_{L^2_t(L^6_x)}  \| \nabla \eta_{\eps} \|_{L^{3}_x} \Big(   \| \partial_{t} u \|_{L^2_t(L^{\infty}_x)} + \|\ell_{\eps} \|_{L^2_t}\|\nabla u\|_{L^{\infty}_t(L^{\infty}_x)}    \Big) \eps^{5/2} \nonumber \\
\leq & \, C (\eps^{5/2} + \eps^{5/2} |\calS_{\eps}^{in}|^{-1/6}).   	\label{I5}
\end{align}
We will now consider $ I_1 $. Recall that $ u $ is a regular solution, we rewrite 
\begin{equation*}
\eta_{\eps} \partial_t u = - \eta_{\eps} u \cdot \nabla u + \eta_{\eps} \Delta u + \eta_{\eps} \nabla p 
\end{equation*} 
\begin{align*}
 \int_{0}^T \int_{\bbR^d} \rho_{\eps} & u_{\eps} \cdot I_1  =   \int_{0}^T \int_{\bbR^d} \rho_{\eps} \eta_{\eps} u_{\eps}\cdot \partial_t u \\
 = & \, - \int_{0}^T \int_{\bbR^d}\eta_{\eps} u_{\eps} \cdot (u \cdot \nabla) u - \eta_{\eps} u_{\eps} \cdot \Delta u + \eta_{\eps} u_{\eps} \cdot \nabla p \\
 = &\, \int_{0}^T   \int_{\bbR^d} \eta_{\eps} u_{\eps} \otimes u : \nabla u + u_{\eps} \otimes u : \nabla \eta_{\eps} \otimes  u  \\
 & \, -2 \int_0^T \int_{\bbR^d} \eta_{\eps} Du_{\eps}: Du + \frac{u_{\eps} \otimes \nabla \eta_{\eps} + \nabla \eta_{\eps} \otimes u_{\eps} }{2} : Du - u_{\eps} \cdot \nabla \eta_{\eps} (p- \bar{p}^{\eps})
\end{align*}
where $ \bar{p}^{\eps} = p(t,h_{\eps}(t)) $.
We can now rewrite the remainder using the above computations and deduce 
\begin{align}
Rest_{\eps}^1 = & \, -2 \sum_{i = 1}^{5} \int_{0}^T \int_{\bbR^d} \rho_{\eps} u_{\eps} \cdot I_i  - 2 \int_0^T  \int_{\calF_{\eps}(t)} u_{\eps} \otimes u_{\eps}: \nabla R_{\eps}(u) \nonumber   \\
& \, - 4 \int_0^T \int_{R^d} Du_{\eps}: DR_{\eps} (u)  \nonumber \\
= & \, -2 \sum_{i = 2}^{5} \int_{0}^T \int_{\bbR^d} \rho_{\eps} u_{\eps} \cdot I_i   - 2 \int_{0}^T \int_{\bbR^d} \rho_{\eps} u_{\eps} \cdot I_1 \nonumber \\
& \,- 2 \int_0^T  \int_{\bbR^d} u_{\eps} \otimes u_{\eps}: \nabla R_{\eps}(u)  - 4 \int_0^T \int_{R^d} Du_{\eps}: DR_{\eps} (u) \nonumber  \\
= \, & - 2\sum_{j = 1}^{5} J_j \label{equ:rest1}
 \end{align}
where 
\begin{gather*}
J_1 = \sum_{i = 2}^{5} \int_{0}^T \int_{\bbR^d} \rho_{\eps} u_{\eps} \cdot I_i , 
\\ J_2 =  \int_0^T  \int_{\calF_{\eps}(t)} u_{\eps} \otimes u_{\eps}: \nabla R_{\eps}(u) - \eta_{\eps} u_{\eps} \otimes u : \nabla u \\ 
J_3= - \int_0^T  \int_{\bbR^d}  u_{\eps} \otimes u : \nabla \eta_{\eps} \otimes  u -  \frac{u_{\eps} \otimes \nabla \eta_{\eps} + \nabla \eta_{\eps} \otimes u_{\eps} }{2} : Du \\
J_{4} = 2 \int_{0}^{T}\int_{R^d} Du_{\eps}: DR_{\eps} (u) - \eta_{\eps} Du_{\eps}: Du 	\quad \text{ and } \quad J_{5} = - \int_{0}^{T}\int_{R^d}  u_{\eps} \cdot \nabla \eta_{\eps} p.
\end{gather*}
Inequalities \eqref{I2}-\eqref{I3}-\eqref{I4}-\eqref{I5} imply that
\begin{equation}
\label{J1}
|J_1| \leq \int_{0}^{T} C \left(\int_{\bbR^d} \rho_{\eps}|u_{\eps} - R_{\eps}(u)|^2\right)^{1/2} +  C( m_{\eps} + \eps^{5/2}|\calS_{\eps}^{in}|^{-1/6})
\end{equation}
To tackled $ J_2 $, we start by rewriting it as

\begin{align}
J_2 = & \,  \int_0^T  \int_{\calF_{\eps}(t)} u_{\eps} \otimes u_{\eps}: \nabla R_{\eps}(u) - \eta_{\eps} u_{\eps} \otimes u : \nabla u \nonumber \\
= & \,  \int_0^T \int_{\calF_{\eps}(t)} u_{\eps} \otimes u_{\eps}: (\nabla R_{\eps}(u)- \eta_{\eps} \nabla u) + \int_0^T  \int_{\calF_{\eps}(t)}  \eta_{\eps} u_{\eps} \otimes (u_\eps - u) : \nabla u  \nonumber \\
= & \,  \int_0^T \int_{\calF_{\eps}(t)} u_{\eps} \otimes u_{\eps}: (\nabla R_{\eps}(u)- \eta_{\eps} \nabla u) + \int_0^T  \int_{\calF_{\eps}(t)}  \eta_{\eps} (u_{\eps} - u) \otimes (u_\eps - u) : \nabla u \nonumber   \\
& \, + \frac{1}{2}\int_0^T \int_{\calF_{\eps}(t)} \nabla \eta_{\eps} \cdot u \otimes (u_{\eps} - u ): \nabla u.  \label{J2:dec}
\end{align} 
Before estimate the right hand side of the above equality, let us recall from \eqref{est:lp:2} and \eqref{est:nabla:2} the following bounds holds 
\begin{align*}
\| R_{\eps}(u) - \eta_{\eps} u \|_{L^{2}_{x}} \leq  \, C \eps^{3/2} \|u\|_{L^{\infty}_{x}}(1 + \|\nabla \eta_{\eps}\|_{L^3_x}),
\end{align*}
%
%
%
%
%
\begin{align}
\| \nabla R_{\eps}(u) - \eta_{\eps} \nabla u \|_{L^{3/2}_{x}} \leq     C \eps^{2} \| u\|_{W^{1, \infty}_x},\nonumber
\end{align}
and
\begin{align*}
\|\sqrt{\eta_{\eps}}(u_{\eps} - u)\|_{L^2(\calF_{\eps}(t))}^2 \leq & \, \| u_{\eps}\|_{L^{2}(\calF_{\eps}(t))}^2 - 2 \int_{\calF_{\eps}(t)} u_{\eps} \cdot R_{\eps}(u) + \|R_{\eps}(u)\|_{L^{2}(\calF_{\eps}(t))}^2 \\
& \, - 2 \int_{\calF_{\eps}(t)} u_{\eps}\cdot ( \eta_{\eps} u - R_{\eps}(u)) + \int_{\calF_{\eps}(t)}  u \cdot (\eta_{\eps} u - R_{\eps}(u))  \\
& \, + \int_{\calF_{\eps}(t)} R_{\eps}(u) \cdot (\eta_{\eps} u - R_{\eps}(u)) \\
\leq & \, \| u_{\eps} - R_{\eps}(u)\|_{L^2(\calF_{\eps}(t))}^2 + C(\| u_{\eps}\|_{L^2_x} + \| u \|_{L^2_x} + \|R_{\eps}(u)\|_{L^2_x}) \| R_{\eps}(u) - \eta_{\eps} u \|_{L^{2}_{x}} \\
\leq & \, \|  u_{\eps} - R_{\eps}(u)\|_{L^2(\calF_{\eps}(t))}^2 + C \eps^{3/2}.
\end{align*}
%
%
%
%
%
Using equality \eqref{J2:dec}, we deduce
\begin{align}
|J_2| \leq & \|u_{\eps}\|_{L^{2}_t(L^{6}_x)}^2 \| \nabla R_{\eps}(u)- \eta_{\eps} \nabla u \|_{L^{\infty}_t(L^{3/2}_x)} + \|\sqrt{\eta_{\eps}}(u_{\eps} - u)\|_{L^{2}_t(L^2(\calF_{\eps}(t)))}^2\| \nabla u\|_{L^{\infty}_t(L^{\infty}_x)}  \nonumber \\
& \, + \eps^{3/2}\|\nabla \eta_{\eps} \|_{L^{3}_x} (\|u_{\eps}\|_{L^2_t(L^{6}_x)} + \|u\|_{L^2_t(L^{6}_x)}) \|\nabla u\|_{L^2_t(L^{\infty}_x)} \nonumber \\
\leq  & \,  C \eps^2 + C \int_0^T \int_{\bbR^d} \rho_{\eps} | u_{\eps} - R_{\eps}(u)|^2 + C \eps^{3/2}.   \label{J2}
\end{align} 
After applying some H\"older inequality, the $ J_3 $ term is bounded as follows.
\begin{align}
| J_3| \leq \eps^{3/2} \|u_{\eps}\|_{L^{2}_t(L^{6}_x)}\|\nabla \eta_{\eps}\|_{L^{3}_x}\left(\| u \|_{L^{\infty}_t(L^{\infty}_x)} \|u\|_{L^2_t(L^{\infty}_x)} +  \|Du\|_{L^{2}_t(L^{\infty}_x)} \right) \leq C \eps^{3/2}. \label{J3}
\end{align}
We now tackle $ J_4 $.
\begin{align} 
|J_4| \leq & \,  \left|   2\int_{0}^{T}\int_{R^d} Du_{\eps}: DR_{\eps} (u) - \eta_{\eps} Du_{\eps}: Du \right|  \nonumber \\ 
\leq & \, \| Du_{\eps}\|_{L^{2}_t(L^{2}_{x})} \| DR_{\eps} (u) - \eta_{\eps} Du \|_{L^{2}_t(L^{2}_{x})}  \nonumber \\
\leq & \, C \eps^{3/2},  \label{J4}
\end{align}
 which follows from the estimate \eqref{est:nabla:2} that reads $ \| DR_{\eps} (u) - \eta_{\eps} Du \|_{L^{2}_{x}}  \leq C \eps^{3/2}\| u\|_{W^{1, \infty}_x} $. 
 Finally we estimate the term $ J_{5} $ as follows 
 \begin{align}
 |J_5 | \leq & \,  \left|  \int_{0}^{T}\int_{R^d}  u_{\eps} \cdot \nabla \eta_{\eps} p \right|\nonumber  \\
\leq   & \,  \eps^{3/2} \| u_{\eps}\|_{L^{2}_t(L^{6}_{x})} \| \nabla \eta_{\eps}\|_{L^3_x} \| p\|_{L^2_t(L^{\infty})} \nonumber  \\
\leq & \, C \eps^{3/2}. \label{J5}
 \end{align}
 If we collect estimates \eqref{J1}-\eqref{J2}-\eqref{J3}-\eqref{J4}-\eqref{J5} and equality \eqref{equ:rest1} we deduce
 \begin{equation}
 \label{Rest1:est}
 |Rest_{\eps}^1|  \leq C \int_{0}^{T} \int_{\bbR^d} \rho_{\eps}|u_{\eps} - R_{\eps}(u)|^2 +  C( m_{\eps} + \eps^{3/2} + \eps^{5/2}|\calS_{\eps}^{in}|^{-1/6}).
 \end{equation}

 Let recall that 
 $$ |\widetilde{Rest_{\eps}}| \leq |Rest^1_{\eps}| + |Rest_{\eps}^2|. $$
 This together with estimates \eqref{Rest1:est} and \eqref{Rest2:est} implies the lemma.

Let now explain how to reproduce the above bounds in dimension two. First of all $ |B_\eps(0)| \leq C \eps^{2}$ and $ \nabla \eta_{\eps} $ is bounded in $ L^{p} $ for $ p \in [1,2] $ and diverge otherwise. Moreover notice the a priori bound $ \|u_{\eps} \|_{L^{2}_t(L^{q}_x)} \leq C_q $ holds for any $ q \in [2,\infty) $ by the Sobolev embedding in dimension two, see Lemma \ref{lemma:sob:emb}. Using this information we can estimate for example 
\begin{align*}
| J_3| \leq  \|u_{\eps}\|_{L^{2}_t(L^{q}_x)} \|\nabla \eta_{\eps}\|_{L^{p}_x}\left(\| u \|_{L^{\infty}_t(L^{\infty}_x)} \|u\|_{L^2_t(L^{\infty}_x)} +  \|Du\|_{L^{2}_t(L^{\infty}_x)} \right) \leq C \eps^{\tilde{\delta}}.
\end{align*}
 where we choose $ p = (\tilde{\delta}+1)/2 $ and $ 1/p + 1/ q = 1 $.

 The bound of all the other terms follows similarly.

\end{proof}

\appendix 

\section{Bogovski\u{\i} operator}
\label{app:B}

This appendix is dedicate to recall some facts about the Bogovski\u{\i} operators. Let recall that a Bogovski\u{\i} operator is a left inverse of the divergence on $ \tilde{L}^{p} $ which is the space of $ L^{p} $ functions with integral zero. Due to the non-uniqueness of this operator, we choose $ B_{1} $ to satisfy the following extra property.

\begin{Theorem}

There exists a Bogovski\u{\i} operator $ B_{1} $ such that 
\begin{equation*}
 B_{1}: \tilde{L}^{p}(B_2(0)\setminus B_1(0)) \lra W^{1,p}_{0}(B_2(0)\setminus B_1(0)) 
\end{equation*}
and it is linear and continuous for any $  1 < p < +\infty$,
\begin{gather*}
\DIV(B_{1}[f]) = f \text{ for any } f \in \tilde{L}^p(B_2(0)\setminus B_1(0)) \\ \text{ and } \quad  \|B_{B_2(0)\setminus B_1(0)}[f]\|_{L^{\infty}(\Omega)} \leq \|f\|_{L^2(B_2(0)\setminus B_1(0))}.
\end{gather*}
Moreover for any vector field $ F \in L^{p}(B_2(0)\setminus B_1(0)) $ such that $ F\cdot n = 0 $ on $ \partial B_2(0)\cup \partial B_1(0) $, it holds 
\begin{equation*}
\|B_1[\DIV(F)]\|_{L^{p}(B_2(0)\setminus B_1(0))} \leq \|F\|_{L^p(B_2(0)\setminus B_1(0))}.
\end{equation*}

\end{Theorem}
We refer to subsection 3.3.1.2 of \cite{NS} for a proof of the above theorem and for more details.

Let recall that $ B_{\eps}[f](x) = \eps B_1[f(\eps y)](x/\eps) $, in particular it satisfies the following uniform estimates. 

\begin{Lemma} 

The operator $ B_{\eps } $ is a Bogovski\u{\i} operator in $ B_{2\eps}(0) \setminus B_{\eps}(0 ) $. Moreover there exists a constant $ C $ independent of $ \eps $ such that 
 \begin{equation*}
\| B_{\eps}[f] \|_{W^{1,p}(B_{2\eps}(0) \setminus B_{1}(0 ))} \leq C\| f \|_{L^{p}(B_{2\eps}(0) \setminus B_{1}(0 ))},
\end{equation*}
for any $ f \in \tilde{L}^{p}(B_{2\eps}(0) \setminus B_{1}(0 )) $. And 
\begin{equation*}
\|B_{\eps}[\DIV(F)]\|_{L^{p}(B_{2\eps}(0)\setminus B_{\eps}(0))} \leq \|F\|_{L^p(B_{2\eps}(0)\setminus B_{\eps}(0))},
\end{equation*}
for any vector field $ F \in L^{p}(B_{2\eps}(0)\setminus B_{\eps}(0)) $ such that $ F\cdot n = 0 $ on $ \partial B_{\eps}(0)\cup \partial B_{\eps}(0) $.

\end{Lemma}
 
 The proof of the above result is consequence of the scaling.

\section{A remark on Sobolev embeddings}

In this section we show that Soblev embeddings $ W^{1,2}(\bbR^2 ) \subset L^{p}(\bbR^2) $  for $  p \in [2, \infty) $ holds when we replace  $ W^{1,2}(\bbR^2) $ with $ 	\dot{H}^1(\bbR^2) \cap L^{2}(\calF_{\eps}(t))$.     

\begin{Lemma}
\label{lemma:sob:emb}
Let $ u \in H^{1}(\bbR^2) $ then $ p \in [2, \infty) $, we have
\begin{equation*}
 \| u \|_{L^{p}(\bbR^2)} \leq C(\|\nabla u\|_{L^2(\bbR^2)} + \| u \|_{L^{2}(\calF_{\eps}(t))}).
\end{equation*}

\end{Lemma}

\begin{proof}

The estimate is well-known if in the right hand side we replace $ \| u \|_{L^{2}(\calF_{\eps}(t))} $ by $ \| u \|_{L^{2}(\bbR^2)} $. Let us show that 
\begin{equation}
\label{unicorno}
\| u \|_{L^{2}(\bbR^2)} \leq C(\| \nabla u \|_{L^{2}(\bbR^2)} + \| u \|_{L^{2}(\calF_{\eps}(t))}).
\end{equation}

To see this let recall that $ \calS_{\eps}(t) \subset B_{\eps}(h_{\eps}(t)) $. By translation invariant of the norms we can assume $ h_{\eps}(t) = 0 $. Introduce the space 
\begin{equation*} 
X =  \left\{ v \in H^{1}(B_{2}(0)) \text{ such that } \bar{v} = \frac{1}{|B_{2}(0)\setminus B_1(0)|} \int_{|B_{2}(0)\setminus B_1(0)|}  v  = 0       \right\}.
\end{equation*}
The Poincar\'e inequality in $ X $ implies the existence of a constant $ C_X $ such that 

\begin{equation*}
\| v \|_{L^2(B_{2}(0))} \leq C_X \|\nabla v\|_{L^2(B_{2}(0))}
\end{equation*}
We deduce that 
\begin{align*}
\| u \|_{L^2(B_{2}(0))} \leq \left\| u - \bar{ u} \right\|_{L^2(B_{2}(0))} + \left\| \bar{ u} \right\|_{L^2(B_{2}(0))} \leq C_X  \|\nabla u\|_{L^2(B_{2}(0))} +  C \| u \|_{L^2(B_{2}(0))\setminus B_{1}(0)},
\end{align*}
which implies \eqref{unicorno}.

\end{proof}

\subsection*{Acknowledgment}
The research of M.B. is supported by the NWO grant OCENW.M20.194. The research of \v S.N. leading to these results has received funding from the Czech Sciences Foundation (GA\v CR), 22-01591S.  Moreover,  \v S. N. has been supported by  Praemium Academiae of \v S. Ne\v casov\' a. CAS is supported by RVO:67985840.










\end{document}